\documentclass[11pt]{amsart}

\usepackage{amsmath,amssymb,lmodern}

\usepackage{graphicx}

\usepackage[margin=1.3in]{geometry}

\usepackage[T1]{fontenc}
\usepackage{amsmath,amssymb,amsthm}
\usepackage{enumitem}
\usepackage{xcolor}

\usepackage{float}

\usepackage[colorlinks=true,linkcolor=blue,citecolor=blue,urlcolor=blue]{hyperref}

\newtheorem{theorem}{Theorem}[section]
\newtheorem{corollary}[theorem]{Corollary}
\newtheorem{lemma}[theorem]{Lemma}

\theoremstyle{definition}
\newtheorem{definition}[theorem]{Definition}

\theoremstyle{remark}
\newtheorem{remark}[theorem]{Remark}

\newcommand{\Th}{\operatorname{Th}}

\usepackage[mathlines]{lineno}

\title[The Failure of Stable Composition in Simple Theories]{The Failure of Stable Composition for Equivalence Relations in Simple Theories}

\author{Mostafa Mirabi}
\address{\newline The Taft School, Watertown, CT 06795, USA and \newline  Wesleyan University, Middletown, CT 06459, USA}
\email{mmirabi@wesleyan.edu}
\urladdr{https://sites.google.com/site/mostafamirabi}


\begin{document}

\begin{abstract}
Casanovas and Potier proved that algebraic quantification preserves stability of formulas. They also gave a nonsimple example, answering a question of Laskowski, showing that the algebraicity hypothesis cannot simply be replaced by NFCP, and asked whether a similar example exists in a simple theory. We give such an example in elementary form. The edge-set structure of the random bipartite graph has two definable equivalence relations, both stable and NFCP, whose relational composition has the order property. The resulting theory is simple and $\aleph_0$-categorical. We also prove a formal sharpness observation: every formula in every first-order theory is an existential composition of two stable NFCP formulas algebraic in the two outer variables. Consequently, closure of stable NFCP formulas under this composition characterizes stability of the whole theory.
\end{abstract}

\maketitle

\section{Introduction}

A formula \(\varphi(x,y)\) is \emph{stable} if it does not have the order property, that is, if there are no sequences \((a_i:i<\omega)\) and \((b_j:j<\omega)\) such that $$ \models \varphi(a_i,b_j) \, \Longleftrightarrow \,  i<j.$$
The stable forking conjecture asks whether, in every simple first-order theory, every forking extension is witnessed by an instance of a stable formula.  
 We also recall that a formula \(\varphi(x,y)\) has the finite cover property, FCP, if there are arbitrarily large finite minimally inconsistent families of instances \(\varphi(x,b)\); otherwise \(\varphi\) has NFCP.

In \cite{CP}, Casanovas and Potier prove that a theory has stable forking if and only if its expansion by imaginaries has stable forking. A key step is their algebraic quantification lemma: if \(\varphi(x,y)\) is stable and \(\theta(v,x)\) implies that, for each \(v\), there are exactly \(n\) possible \(x\), then  
$
  \psi(v,y):=\exists x\bigl(\theta(v,x)\wedge \varphi(x,y)\bigr)
$  
is stable \cite[Lemma 1.2]{CP}. In their final section, they ask how far the algebraicity hypothesis can be weakened. They give a nonsimple example, answering a question of M. C. Laskowski, in which \(E(x,y)\) has NFCP and \(F(y,z)\) is stable but \(\exists y(E(x,y)\wedge F(y,z))\) is unstable; their example has the strict order property, and they ask whether a similar example exists in a simple theory \cite[Section 3]{CP}.

The purpose of this note is to answer that question. The example is the edge set of the random bipartite graph. Two edges are \(E\)-equivalent if they have the same left endpoint, and \(F\)-equivalent if they have the same right endpoint. Both \(E\) and \(F\) are equivalence relations, hence stable and NFCP. But the composition 
$
  \exists y\bigl(E(x,y)\wedge F(y,z)\bigr)
$\,
says that the left endpoint of \(x\) is adjacent to the right endpoint of \(z\), and therefore recovers the random bipartite edge relation. This formula has the order property. Pairs of equivalence relations have been studied from the point of view of stability and classification theory; see, for example, Toffalori \cite{Toffalori}. The present note concerns the different, local preservation question for stable formulas raised in \cite{CP}.

The example is close to optimal for this line of argument. In Section \ref{sec:universal}, we prove a formal anti-preservation statement: every formula \(\rho(x,y)\) can be written as an existential composition of two formulas which are stable, NFCP, and one-sided algebraic in the two outer variables. Thus, in an unstable theory, no preservation theorem for existential relational composition can follow merely from stability, NFCP, equivalence-relation-like behavior, or one-sided algebraicity in the two unquantified outer variables. The algebraicity in Casanovas-Potier's lemma is algebraicity of the quantified variable over the outer variable, and this direction is essential.

AI Use Statement:  ChatGPT 5.5 was used for proofreading, language editing, and preparing the LaTeX code for Figure 1. It was not used to conduct the research.

\section{The random bipartite graph example}\label{sec:example}

Let \(G=(U,V,R)\) be the countable random bipartite graph. Thus \(U\) and \(V\) are disjoint sorts and \(R\subseteq U\times V\) satisfies the usual extension axioms: for any finite disjoint \(A_0,A_1\subseteq V\), there is \(u\in U\) adjacent to all elements of \(A_0\) and to no element of \(A_1\), and symmetrically with the roles of \(U\) and \(V\) interchanged. Equivalently, \(G\) is the Fraisse limit of the class of finite bipartite graphs. The theory of \(G\) is simple. To see this, let \(H=(D,S)\) be the countable random graph and fix \(c\in D\). Define
\[
  U'=\{u\in D:S(u,c)\} \text{ and }
  V'=\{v\in D:v\neq c\wedge\neg S(v,c)\}.
\]
Retaining only the relation \(S\!\upharpoonright_{U'\times V'}\) gives a random bipartite graph. Indeed, the extension axioms for \(H\) realize any prescribed finite adjacency pattern to vertices in the opposite part, together with adjacency to \(c\) for a new vertex in \(U'\), or nonadjacency to \(c\) for a new vertex in \(V'\). The theory of the random graph is simple; see, for example, the opening of \cite[Section 7]{Conant}. Naming a parameter and passing to an interpreted structure preserve simplicity. Hence \(\Th(G)\) is simple.

Let
\[
  W:=\{(u,v)\in U\times V:R(u,v)\}
\]
be the edge sort. On \(W\), define two binary relations
\[
  E((u,v),(u',v')) \quad\Longleftrightarrow\quad u=u',
\]
and
\[
  F((u,v),(u',v')) \quad\Longleftrightarrow\quad v=v'.
\]
Equivalently, in the definable edge-sort expansion, \(E\) is equality of source and \(F\) is equality of target. Let \(M=(W,E,F)\), considered as a one-sorted structure in the language \(\{E,F\}\), and let \(T_{EF}=\Th(M)\).

\begin{remark}\label{rem:simple-reduct}
The structure \(M\) is interpretable, without parameters, in the definable edge-sort expansion of the random bipartite graph. Therefore \(T_{EF}\) is simple: any tree-property configuration in the interpreted structure would translate to one in the ambient simple theory. It is also \(\omega\)-categorical, since it is a finite-language reduct of an \(\omega\)-categorical definable expansion. In \(T_{EF}^{\mathrm{eq}}\), the quotient sorts \(W/E\) and \(W/F\), together with the relation saying that two classes meet, recover the original random bipartite graph. Thus the construction does not add model-theoretic complexity; it only presents the random bipartite graph through incidences of edges with endpoints.
\end{remark}

We first prove the elementary stability fact used below.

\begin{lemma}\label{lem:eqstable}
Every equivalence-relation formula is stable and has NFCP. More precisely, if \(H(x,y)\) defines an equivalence relation, then \(H(x,y)\) is stable and has no finite cover property.
\end{lemma}

\begin{proof}
Suppose first that \(H(x,y)\) has the order property, witnessed by \((a_i:i<\omega)\) and \((b_j:j<\omega)\), so that \(H(a_i,b_j)\) holds if and only if \(i<j\). From the column \(j=2\), we get \(H(a_0,b_2)\) and \(H(a_1,b_2)\), hence \(H(a_0,a_1)\). From the column \(j=1\), we get \(H(a_0,b_1)\) and \(\neg H(a_1,b_1)\), contradicting transitivity and symmetry of \(H\). Hence \(H\) is stable.

For NFCP, a finite family of instances \(H(x,b_i)\) is just a finite family of equivalence classes. If such a family is inconsistent and minimal with this property, then two of the classes are distinct, and the two corresponding instances are already inconsistent. Thus there are no arbitrarily large finite minimally inconsistent families.
\end{proof}

\begin{theorem}\label{thm:main-example}
The complete simple theory \(T_{EF}\) has two formulas \(E(x,y)\) and \(F(y,z)\), both stable and NFCP, such that
\[
  \Gamma(x,z):=\exists y\bigl(E(x,y)\wedge F(y,z)\bigr)
\]
is unstable.
\end{theorem}

\begin{proof}
By construction, \(E\) and \(F\) are equivalence relations on \(W\), so they are stable and NFCP by Lemma \ref{lem:eqstable}. It remains to show that \(\Gamma(x,z)\) has the order property.

For \(x=(u,p)\in W\) and \(z=(q,v)\in W\), the formula \(\Gamma(x,z)\) holds if and only if there is an edge \(y\in W\) whose source is \(u\) and whose target is \(v\). Since elements of \(W\) are exactly edges of the bipartite graph, this is equivalent to \(R(u,v)\).

Fix \(n<\omega\). Consider the finite bipartite graph with left vertices $  u_0,\ldots,u_{n-1},q $
and right vertices
$
  v_0,\ldots,v_{n-1},p,
$
all distinct, whose edges include \(R(u_i,v_j)\) exactly when \(i<j\), all edges \(R(u_i,p)\), and all edges \(R(q,v_j)\). Since \(G\) is the random bipartite graph, this finite bipartite graph embeds into \(G\). We may therefore choose such elements in \(G\). Let
$
  a_i:=(u_i,p)\in W$ and $ b_j:=(q,v_j)\in W.
$
Then
\[
  \Gamma(a_i,b_j)
  \quad\Longleftrightarrow\quad
  R(u_i,v_j)
  \quad\Longleftrightarrow\quad
  i<j.
\]
Thus finite half-graphs of arbitrary size occur. By compactness, \(\Gamma(x,z)\) has the order property.
\end{proof}
See Figure 1.
\begin{figure}[H]\label{figure}
\centering
\includegraphics{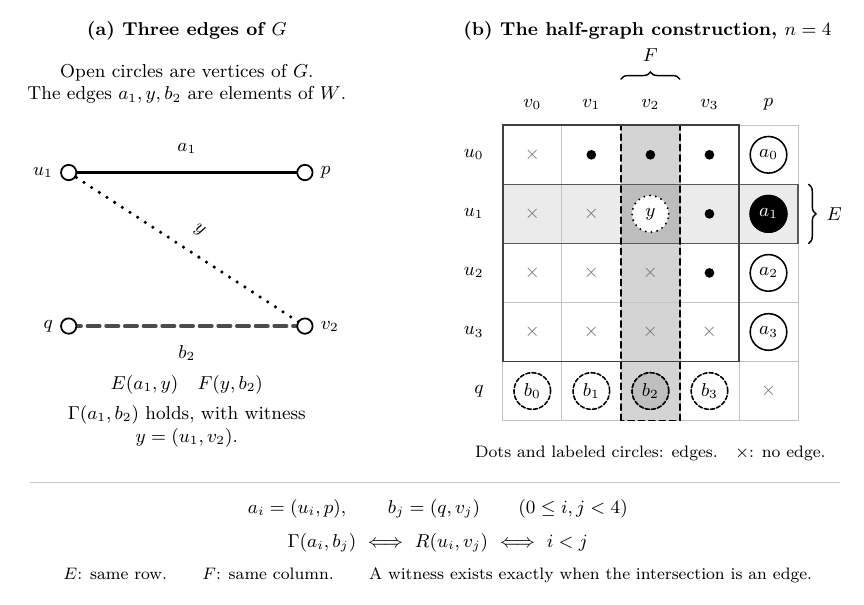}
\caption{The edge-set construction for $n=4$.
Panel (a) shows three edges of $G$, viewed as elements of $W$:
$E(a_1,y)$ and $F(y,b_2)$, so $y$ witnesses $\Gamma(a_1,b_2)$.
Panel (b) records a finite bipartite graph: 
The lightly shaded row with solid boundary marks the $E$-class of $a_1$;
the darker column with dashed boundary marks the $F$-class of $b_2$
within this configuration. Their intersection contains $y$.
The framed upper-left block has $R(u_i,v_j)$ exactly when $i<j$,
giving the same pattern for $\Gamma(a_i,b_j)$.}
\label{fig:stable-composition}
\end{figure}

\begin{corollary}\label{cor:cp-answer}
There is a simple, indeed \(\omega\)-categorical simple, first-order theory with equivalence relations \(E\) and \(F\) such that \(E(x,y)\) and \(F(y,z)\) are stable and NFCP, but \(\exists y(E(x,y)\wedge F(y,z))\) is unstable.
\end{corollary}

\begin{remark}
This is formally close to the example in \cite[Section 3]{CP}. In both examples, the two composing relations are equivalence relations, and hence are stable and NFCP. The difference is that the ambient theory here is simple and \(\omega\)-categorical. 
\end{remark}

\section{A universal anti-preservation principle}\label{sec:universal}

The previous section gives a natural example answering the question from \cite{CP}. We now prove a formal observation showing that stability and NFCP of the two factors, even together with algebraicity in the outer variables, do not guarantee stability of their existential composition in an unstable theory. 

\begin{definition}
Let \(k<\omega\). If \(u\) is a displayed block of variables in a formula \(\varphi(u,v)\), then \(\varphi\) is \emph{\(k\)-algebraic in \(u\)} if
$
  T\models \forall v\,\exists^{\leq k} u\,\varphi(u,v).
$
More generally, for a formula with several displayed blocks, we say that it is \(k\)-algebraic in one specified block if, after fixing all the remaining blocks, there are at most \(k\) realizations of that block. It is \emph{one-sided algebraic} if it is \(k\)-algebraic in one of the displayed variable blocks for some finite \(k\).
\end{definition}

\begin{lemma}\label{lem:algebraic-stable}
If \(\varphi(x,y)\) is \(k\)-algebraic in \(x\), then \(\varphi(x,y)\) is stable and has NFCP.
\end{lemma}

\begin{proof}
Suppose \(\varphi(x,y)\) has the order property, witnessed by \((a_i:i<\omega)\) and \((b_j:j<\omega)\). Choose \(j>k\). Then \(\varphi(a_i,b_j)\) holds for all \(i<j\). Since the fiber \(\varphi(x,b_j)\) has size at most \(k\), two of \(a_0,\ldots,a_{j-1}\), say \(a_i\) and \(a_{i'}\) with \(i<i'<j\), are equal. But at the column \(j=i'\), one has \(\varphi(a_i,b_{i'})\) and \(\neg\varphi(a_{i'},b_{i'})\), contradiction. Thus \(\varphi\) is stable.

For NFCP, consider a finite minimally inconsistent family of instances $\{\varphi(x,b_i):i\in I\}$.

For each \(i\in I\), choose \(a_i\) realizing all instances except possibly \(\varphi(x,b_i)\). Since the whole family is inconsistent, \(a_i\not\models \varphi(x,b_i)\). If \(i\neq j\), then \(a_i\neq a_j\): indeed, \(a_i\models \varphi(x,b_j)\), while \(a_j\not\models \varphi(x,b_j)\). Now fix \(i_0\in I\). The instance \(\varphi(x,b_{i_0})\) contains all \(a_i\) with \(i\neq i_0\), so \(|I|-1\leq k\). Hence every finite minimally inconsistent family has size at most \(k+1\), and \(\varphi\) has NFCP.
\end{proof}

\begin{theorem}\label{thm:universal-factorization}
Let \(T\) be any complete first-order theory and let \(\rho(x,y)\) be any formula. Let \(z=(z_0,z_1)\), where \(z_0\) has the same sorts as \(x\) and \(z_1\) has the same sorts as \(y\). Define
\[
  \alpha(x;z_0,z_1):=(x=z_0)\wedge \rho(z_0,z_1)
 \,\,\, \text{ and }\,\
  \beta(z_0,z_1;y):=(z_1=y).
\]
Then \(\alpha(x,z)\) and \(\beta(z,y)\) are stable and have NFCP. Moreover, \(\alpha\) is \(1\)-algebraic in the outer variable \(x\), \(\beta\) is \(1\)-algebraic in the outer variable \(y\), and $
  \rho(x,y)\equiv \exists z\bigl(\alpha(x,z)\wedge \beta(z,y)\bigr).$
\end{theorem}

\begin{proof}
The formula \(\alpha(x;z_0,z_1)\) is \(1\)-algebraic in \(x\), and hence is stable and NFCP by Lemma \ref{lem:algebraic-stable}. The formula \(\beta(z_0,z_1;y)\) is coordinate equality. To see directly that it is stable, suppose \(\beta(c_i,d_i;b_j)\) holds exactly when \(i<j\). The column \(j=2\) gives \(d_0=b_2=d_1\), while the column \(j=1\) gives \(d_0=b_1\) and \(d_1\neq b_1\), contradiction. For NFCP, a finite family of instances of \(\beta(z,y)\) consists of conditions of the form \(z_1=b\); any minimally inconsistent such family has size at most two. Also, for each \(z\) there is exactly one \(y\) such that \(\beta(z,y)\), so \(\beta\) is \(1\)-algebraic in the outer variable \(y\).

Finally,
\[
\begin{aligned}
  \exists z_0z_1\bigl(&x=z_0\wedge \rho(z_0,z_1)\wedge z_1=y\bigr)
  \quad\Longleftrightarrow\quad \rho(x,y).
\end{aligned}
\]
\end{proof}

\begin{corollary}\label{cor:characterization}
For a complete theory \(T\), the following are equivalent.
\begin{enumerate}[label=\textup{(\arabic*)}]
\item \(T\) is stable.
\item For all stable NFCP formulas \(\alpha(x,z)\) and \(\beta(z,y)\), the formula
$
  \exists z\bigl(\alpha(x,z)\wedge \beta(z,y)\bigr)
$
is stable.
\item The statement in \textup{(2)} holds even under the additional restrictions that \(\alpha\) is \(1\)-algebraic in \(x\) and \(\beta\) is \(1\)-algebraic in \(y\).
\end{enumerate}
\end{corollary}

\begin{proof}
If \(T\) is stable, every formula is stable, so (1) implies (2), and (2) trivially implies (3). Conversely, assume (3). Given any formula \(\rho(x,y)\), Theorem \ref{thm:universal-factorization} writes \(\rho\) as such an existential composition. Hence \(\rho\) is stable. Since \(\rho\) was arbitrary, \(T\) is stable.
\end{proof}

\begin{remark}
Theorem \ref{thm:universal-factorization} is deliberately formal: it shows that any proposed preservation theorem for relational composition must exclude this kind of graph factorization. The random bipartite graph example of Section \ref{sec:example} is less formal and closer to the question in \cite{CP}: the two composing formulas are simply equivalence relations on the same sort.
\end{remark}

\begin{remark}[Relation to stable forking]\label{rem:stable-forking}
The results above do not give a counterexample to the stable forking conjecture. The random bipartite graph is one of the basic tame simple free-amalgamation structures. The point of the example is instead methodological. Casanovas and Potier's proof that stable forking passes to imaginaries uses algebraic quantification precisely at the step where one moves a stable formula through representatives of an imaginary \cite[Proposition 2.1]{CP}. The examples in this note show that this algebraic step cannot in general be replaced by a nonalgebraic stable or NFCP correspondence, even in a simple theory.

In the equivalence-relation example, each of \(E\) and \(F\) is individually very tame: both are stable equivalence relations and both have NFCP. Nevertheless, their relational composition remembers whether an \(E\)-class and an \(F\)-class meet. The pattern of meetings of the two quotient families is a random bipartite graph, and this is exactly where the order property appears.

Thus the algebraicity hypothesis in \cite[Lemma 1.2]{CP} cannot be replaced by stability and NFCP of the composing formulas, even in a simple theory. The two constructions show that adding algebraicity in the outer variables, or instead requiring both factors to be equivalence relations, does not ensure preservation.
\end{remark}

\end{document}